\documentclass[12pt,reqno]{amsart}
\usepackage{geometry}
\usepackage[latin1]{inputenc}
\usepackage[italian,english]{babel}
\usepackage{amsmath, amsfonts, amsthm, xcolor}
\usepackage{paralist}
\numberwithin{equation}{section}
\usepackage{mathtools, mathabx, verbatim}
\usepackage{pgfplots}
\usepackage{mathrsfs}
\usetikzlibrary{arrows}
\definecolor{ududff}{rgb}{0.30196078431372547,0.30196078431372547,1}
\definecolor{yqqqqq}{rgb}{0.5019607843137255,0,0}
\definecolor{xdxdff}{rgb}{0.49019607843137253,0.49019607843137253,1}

\newcommand{\interior}[1]{%
  {\kern0pt#1}^{\mathrm{o}}%
}
\newtheorem{thm}{Theorem}[section]

\newtheorem{cor}[thm]{Corollary}
\newtheorem{prop}[thm]{Proposition}

\newtheorem{defn}[thm]{Definition}
\theoremstyle{definition}
\newtheorem{rem}[thm]{Remark}
\theoremstyle{remark}

\newcommand{\ds}{\displaystyle}

\newcommand{\R}{\mathbb{R}}

\newcommand{\de}{\partial}

\DeclareFontFamily{U}{mathb}{\hyphenchar\font45}
\DeclareFontShape{U}{mathb}{m}{n}{ <-6> matha5 <6-7> matha6 <7-8>
mathb7 <8-9> mathb8 <9-10> mathb9 <10-12> mathb10 <12-> mathb12 }{}
\DeclareSymbolFont{mathb}{U}{mathb}{m}{n}

\DeclareMathAccent{\abxring}{0}{mathb}{"38}

\DeclareFontFamily{U}{mathb}{\hyphenchar\font45}
\DeclareFontShape{U}{mathb}{m}{n}{ <-6> matha5 <6-7> matha6 <7-8>
mathb7 <8-9> mathb8 <9-10> mathb9 <10-12> mathb10 <12-> mathb12 }{}
\DeclareSymbolFont{mathb}{U}{mathb}{m}{n}
\DeclareMathOperator{\dive}{div}

\DeclareMathOperator{\spt}{spt}

{\left\{\begin{array}{@{}l@{}}}{\end{array}\right.}
\patchcmd{\abstract}{\scshape\abstractname}{\textbf{\abstractname}}{}{}
\makeatletter 
\def\@makefnmark{} 
\makeatother 

\title[A monotonicity result for the first Steklov-Dirichlet eigenvalue]{A monotonicity result for the first\\ Steklov-Dirichlet Laplacian 
eigenvalue} 
\author[N. Gavitone, G. Piscitelli]{
	Nunzia Gavitone,  Gianpaolo Piscitelli}
\address{Dipartimento di Matematica e Applicazioni ``R. Caccioppoli'', Universit\`a degli studi di Napoli Federico II \\ Via Cintia, Complesso Universitario Monte S. Angelo, 80126 Napoli, Italy.}
\email{nunzia.gavitone@unina.it}
\address{Dipartimento di Matematica e Applicazioni ``R. Caccioppoli'', Universit\`a degli studi di Napoli Federico II \\ Via Cintia, Complesso Universitario Monte S. Angelo, 80126 Napoli, Italy.}
\email{gianpaolo.piscitelli@unina.it}

\begin{document}
\maketitle
\begin{abstract}
%
In this paper, we  consider  the first Steklov-Dirichlet eigenvalue  of the Laplace operator in annular domains with a spherical hole. We prove a monotoni-\\city result with respect to the hole, when the outer region is  centrally symmetric.

\vspace{.1cm}


\vspace{.1cm}

\end{abstract}

\section{Introduction and main result}
In recent years, the study of eigenvalue problems in holed domains has been the object of much interest. These kind of problems are usually defined in annular domains  with the outer region and the hole satisfying suitable assumptions. 
Specifically, different boundary conditions can be imposed on the outer and inner  boundary and hence several optimization problems can be studied (e.g. Robin-Neumann \cite{PPT}, Neumann-Robin \cite{DP}, Dirichlet-Neumann \cite{AA,AAK}, Steklov-Dirichlet \cite{PPS,GPPS,HLS}, Steklov-Robin \cite{GS}). At the mean time, the optimal placement of an obstacle has been studied, so as to maximize or minimize a prescribed functional (e.g. the Dirichlet heat content \cite{L}, the first Steklov eigenvalue \cite{F}).

In this paper, we consider an eigenvalue problem for the Laplace operator in a suitable annular domain with outer Steklov and inner Dirichlet boundary conditions.


More precisely, let $\Omega_0 \subset \R^n$, $n \ge 2$,  be an open  bounded  set with Lipschitz boundary and $B_r(y)$ be the ball of radius $r>0$, centered at $y$, such that $B_r(y)\subset \abxring{\Omega}_0$. We study the following Steklov-Dirichlet eigenvalue problem 
\begin{equation}\label{eig}
\sigma( \Omega)=\min_{\substack{v\in H^{1}_{\partial B_{r}(y)}(\Omega)\\ v\not \equiv0}}\left\{\int _{\Omega}|Dv|^2\;dx,\,\|v\|_{L^2(\partial \Omega_0)}=1
\right\},
\end{equation}
where  $H^1_{\partial B_{r}(y)}(\Omega)$  denotes the set of Sobolev functions which vanish on the boundary of $B_r(y)$ (see Section 2 for the precise definition). If $u\in H^{1}_{\partial B_{r}(y)}(\Omega)$ is a minimizer of \eqref{eig}, then it satisfies:
\begin{equation}\label{pb}
\begin{cases}
\Delta u=0 & \mbox{in}\ \Omega\vspace{0.2cm}\\
\dfrac{\de u}{\de \nu}=\sigma \,u&\mbox{on}\ \partial\Omega_0\vspace{0.2cm}\\ 
u=0&\mbox{on}\ \partial B_{r}(y), 
\end{cases}
\end{equation}
where $\Omega$ is the annular domain $\Omega=\Omega_0 \setminus\overline{B_r(y)}$, $\sigma \in \R$ and  $\nu $ is the  outer unit normal to  $\partial\Omega_0$.


In \cite{PPS} (see also \cite{D, HLS}), the authors prove that the minimum in \eqref{eig} is achieved by a function $u\in H^1_{\partial B_{r}(y)}(\Omega)$, which is a weak solution to problem \eqref{pb} with constant sign in $\Omega$, and that $\sigma(\Omega)$ is simple.
Furthermore, the authors prove that, keeping the measure of $\Omega$ and the radius of the inner ball $r$ fixed, $\sigma(\Omega)$ is maximized, among quasi-spherical sets, when $\Omega$ is a spherical shell, that is when $\Omega_0$ is a ball with the same center of the hole. On the other hand, in \cite{GPPS} the authors  extend this result  to a class of annular sets with a suitable convex  outer domain $\Omega_0$. 

More properties are known  when $\Omega$ is an eccentric spherical shell, that is,  when the outer domain $\Omega_0$  is a ball not necessarly  centered at  the same point of the spherical hole.  
In \cite{VS},  the authors  study the optimal placement of the hole in eccentric spherical shell $\Omega$ so that $\sigma(\Omega)$ is maximized when the outer ball and the inner radius are fixed. If $n \ge3$, they  prove that $\sigma(\Omega)$ achieves the maximum when the two balls are concentric. Subsequently, this result has been also proved for any dimension $n \ge 2$ in \cite{F}, by using different proofs (see also \cite{S} for an analogous result in two-points homogeneous spaces).

Moreover, by performing numerical experiments, the authors in \cite{HLS} exhibit that $\sigma(\Omega)$ is  monotone decreasing with respect to the distance between the centers of the two disks. 
Our aim is to prove that this monotonicity property holds in any dimension and  in a more general setting.

Through this paper, we assume that the  outer  domain $\Omega_0$ verifies the following hypotheses. 
\begin{defn}
\label{set}
Let  $\Omega_0 \subset \R^n$, $n \ge 2$, be an open, bounded set with Lipschitz boundary and centrally symmetric with respect to $x_0\in \abxring{\Omega}_0$, that is there exists $x' \in \overline{\Omega_0} $ such that $\frac 12 \left( x+x'\right)=x_0$   for any $x \in \overline{\Omega_0}$. 
\end{defn} 

Let $r>0$ be fixed,  our first  question is 
\begin{itemize}
\item[$(\mathcal Q_1)$]
Where we have to place the center of the spherical hole with fixed radius $r$ in order to maximize $\sigma(\Omega)$? 
\end{itemize}
To give an answer, let $w \in \R^n$ be a unit  vector and let us consider the holes $B_r(t)$ with the centers on the $w$-direction:
\begin{equation}
\label{ph}
B_r(t)=\left\{x(t)=x +tw, \,x \in B_r(x_0)\right\} \quad 0 \le t <\rho_w(x_0,r)-r,
\end{equation}
where 
\begin{equation}
\label{row}
\rho_w=\rho_w(x_0,r)=\sup\{t>0 \colon B_r(x_0+tw)\subset \abxring{\Omega}\}. 
\end{equation}
We stress that $\rho_w$ is the distance between $x_0$ and the center of the farthest ball from $x_0$, well contained in $\Omega_0$, in the direction $w$ passing at $x_0$.  
Then we consider the following type of  annular domains
\begin{equation}
\label{sett}
 \Omega(t)=\Omega_0 \setminus \overline{B_r(t)},
\end{equation}
where $\Omega_0$ verifies the assumption of Definition \ref{set}. We stress that  when $\Omega_0$ is a ball,  then the sets $\Omega(t)$ are eccentric spherical shells. Furthermore, for any $0\le t < \rho_w(x_0,r)-r$, we denote by $\sigma(t)$ the first Steklov-Dirichlet eigenvalue of the Laplacian in $\Omega(t)$, 
that is 
\begin{equation}\label{eigt}
\sigma( t)=\min_{\substack{v\in H^{1}_{\partial B_{r}(t)}(\Omega(t))\\ v\not \equiv0}}\left\{\int _{\Omega(t)}|Dv|^2\;dx,\ \,\|v\|_{L^2(\partial \Omega_0)}=1
\right\}.
\end{equation}
Our second questions is
\begin{itemize}
\item[$(\mathcal Q_2)$] Is $\sigma(t)$ decreasing with respect to $t$?
\end{itemize}
Our main result gives an answer to both questions $(\mathcal Q_1)$ and $(\mathcal Q_2)$.
\begin{thm}
\label{mono}
Let $\Omega_0$ be as in Definition \ref{set}, $w \in \R^n$ be a unit vector, $B_r(t)$, $\rho_w$, $\Omega(t)$ and $\sigma(t)$ be defined as in \eqref{ph}, \eqref{row}, \eqref{sett} and \eqref{eigt}, respectively. Then, $\sigma( t)$ is strictly monotone decreasing with respect to $t\in [0,\rho_w-r)$.
\end{thm}
As an  immediate consequence of our main result, we obtain that, in order to maximize $\sigma(t)$, the  hole has to be centered at the  symmetry point of $\Omega_0$, when the inner radius is fixed.  

To prove Theorem \ref{mono}, we use a shape derivative approach. In particular we compute the first and the second domain derivatives of the first Steklov-Dirichlet eigenvalue. We emphasize that our result implies that the monotonicity property holds also for eccentric spherical shell in any dimension and in particular in two dimensions, as suggested by the numerical computation contained in  \cite{HLS}. 

Finally, we describe the outline of the paper. In Section \ref{prelim}, we summarize some results about the first Steklov-Dirichlet eigenvalue problems. 
In Section \ref{first}, we compute the first shape derivative of $\sigma(t)$, observing that a stationary set occurs when the center of the hole coincides with the center of symmetry of the outer region.  Finally, in Section \ref{second}, we compute the second shape derivative of $\sigma(t)$ and prove the main result.

\section{The Steklov-Dirichlet Laplacian eigenvalue problem
}
\label{prelim}
In this Section, we study the Steklov-Dirichlet Laplacian eigenvalue problem both on fixed domain $\Omega$ and on a parameter-dependent domain $\Omega (t)$, respectively. 
\subsection{Foundation of the problem}
Let $\Omega_0\subset\R^n$ be an open, bounded set with Lipschitz boundary and such that  $B_{r}(y)\subset\abxring\Omega_0$, where   $B_r(y)$ is the ball of radius $r>0$ centered at $y$. 
Let us  consider the annular domain $\Omega :=\Omega_0\setminus \overline{ B_{r}(y)}$. 
In what follows we  denote the set of Sobolev functions on $\Omega$ vanishing on $\partial B_{r}(y)$ by $H^1_{\partial B_{r}(y)}(\Omega)$, that is (see \cite{ET}) the closure in $H^1(\Omega)$ of the following set

\begin{equation*}
C^\infty_{\partial B_{r}(y)} (\Omega):=\{ u|_{\Omega}  \ | \ u \in C_0^\infty (\R^n),\ \spt (u)\cap \partial B_{r}(y)=\emptyset \}.  
\end{equation*}
It is known (see for instance \cite{A,D,P}) that the spectrum of the Steklov-Dirichlet eigenvalue problem \eqref{pb} for the Laplace operator in $\Omega$ is discrete and that the sequence of eigenvalues can be ordered as follows
\[
0<\sigma(\Omega)\le \sigma_2(\Omega)\le\sigma_3(\Omega)\le \ldots \nearrow +\infty.
\]
In  particular, the first eigenvalue $\sigma(\Omega)$ has the variational characterization \eqref{eig}, see for instance  \cite{D,PPS,HLS}. 
Moreover in \cite{PPS} (see also \cite{HLS}), the authors prove the following result.
\begin{prop}
There exists a function $u\in H^1_{\partial B_{r}(y)}(\Omega)$ which achieves the minimum  in \eqref{eig} and is a weak solution to the problem \eqref{pb}. Moreover $\sigma(\Omega)$ is simple and the first eigenfunctions have constant sign in $\Omega$. 
\end{prop}
In order to prove our main result, we need to compute the first and the second shape derivative of $\sigma(\Omega)$ and, to study these derivatives, we need to consider a family of domains approaching to $\Omega$. In our case, we get the desired monotonicity result by studying the behavior of the involved quantities on the family of domains obtained by moving the hole in a fixed direction.  

In what follows we fix notation and recall some preliminary results.
Let $\Omega_0$ be as in Definition \ref{set}. For the reader convenience, from now on, we assume that $x_0=0$. In the sequel, we  denote by $B_r$ the ball centered at the origin with radius $r$ and we set $\Omega=\Omega_0 \setminus \overline{B_r}$ (see the figure below). 
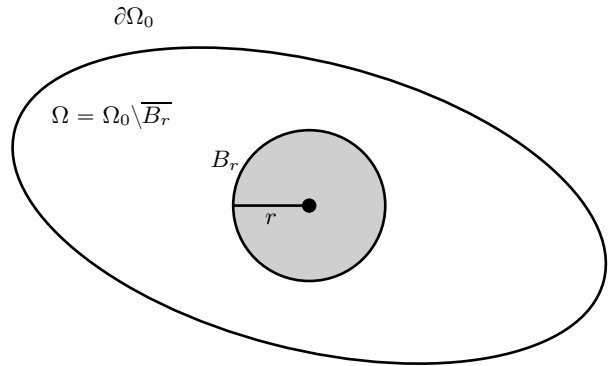
\begin{center}
\begin{tikzpicture}[line cap=round,line join=round,>=triangle 45,x=1cm,y=1cm]
\clip(-4,-2.7) rectangle (4,2.7);
\draw [line width=1.pt,fill=black,fill opacity=0.19] (0,0) circle (1cm);
\draw [rotate around={165.4571281701613:(0,0)},line width=1.pt] (0,0) ellipse (4cm and 1.9cm);
\draw [line width=1pt] (0,0)-- (-1,0);
\begin{scriptsize}
\draw[fill=black] (0,0) circle (2.5pt);
\draw[color=black] (-2.3,2.5) node {$\partial \Omega_0 $};
\draw[color=black] (-2.6,1.2) node {$\Omega= \Omega_0\setminus{\overline{B_r}}$};
\draw[color=black] (-.5,-.17) node {$r $};
\draw[color=black] (-1.1,.6) node {$B_r $};
\end{scriptsize}
\end{tikzpicture} \end{center}

In order to study the Steklov-Dirichlet eigenvalue problem on the annular domain $\Omega (t)$ defined in \eqref{sett}, we define a suitable smooth vector field  which  move the  hole $B_r$ in a given direction $w$ keeping the boundary of $\Omega_0$ and the inner radius $r$ fixed. Hence, the perturbed holes have the form described in \eqref{ph}. We consider the following variational field in $\R^n$:
\begin{equation}
\label{campo}
V(x)=w\varphi(x), 
\end{equation}
where $\varphi \in C_0^\infty(\Omega_0)$ is a cut-off function such that $\varphi(x)=1$ on $\overline {B_r}$.
Consequently, the perturbed annular domains $\Omega(t)$, defined in \eqref{sett}, can be seen as 
\[
\Omega(t)=\{x(t)=x+tV(x),\,\, x\in \Omega\}\quad  t\in[0,\rho_w-r),
\]
where $\rho_w$ is defined in \eqref{row}. Let us observe that $\partial \Omega(t)=\partial \Omega_0 \cup \partial B_r(t)$ and that $\Omega(t)$ is centrally symmetric  if and only if $t=0$ and the center of symmetry is the origin. Also refer to the figure below.
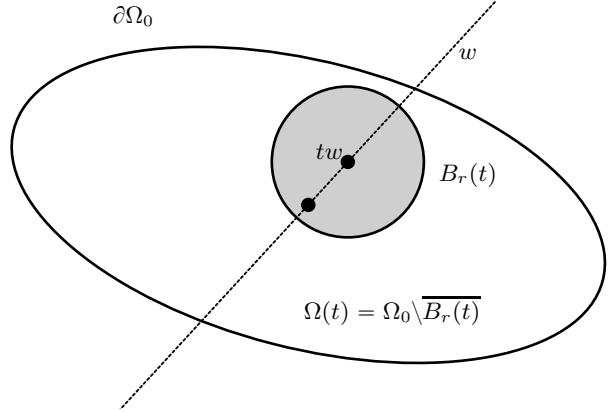
\begin{center}
\begin{tikzpicture}[line cap=round,line join=round,>=triangle 45,x=1cm,y=1cm]
\clip(-4,-2.7) rectangle (4,2.7);
\draw [line width=1.pt,color=black,fill=black,fill opacity=0.19] (0.52,0.57) circle (1cm);
\draw [rotate around={165.4571281701613:(0,0)},line width=1.pt] (0,0) ellipse (4cm and 1.9cm);\draw [line width=.7pt,dash pattern=on 1pt off 1pt,domain=-8.56:8.32] plot(\x,{(-0--0.57*\x)/0.52});
\begin{scriptsize}
\draw[fill=black] (0,0) circle (2.5pt);
\draw [fill=black] (0.52,0.57) circle (2.5pt);
\draw[color=black] (0.28,.7) node {$tw$};
\draw[color=black] (-2.3,2.5) node {$\partial \Omega_0 $};
\draw[color=black] (2.1,2) node {$w $};
\draw[color=black] (2.1,.4) node {$B_r(t)$};
\draw[color=black] (1.1,-1.4) node {$\Omega(t)=\Omega_0 \setminus \overline{B_r(t)}$};
\end{scriptsize}
\end{tikzpicture}
\end{center}

For any $t\in[0, \rho_w-r)$, let $\sigma(t)$ be the first Steklov-Dirichlet eigenvalue \eqref{eigt} of the Laplacian in $\Omega(t)$ and $u^t$ be the corresponding normalized and positive eigenfunction. The first eigenvalue admit the variational characterization \eqref{eigt} and the eigenfunction $u^t\in H^{1}_{\partial B_{r}(t)}(\Omega(t))$ is a solution to the following eigenvalue problem
\begin{equation}\label{pt}
\begin{cases}
\Delta u^t=0 & \mbox{in}\ \Omega(t)\vspace{0.2cm}\\
\dfrac{\de u^t}{\de \nu}=\sigma (t)u^t&\mbox{on}\ \partial\Omega_0\vspace{0.2cm}\\ 
u^t=0&\mbox{on}\ \partial B_{r}(t). 
\end{cases}
\end{equation}
By a little abuse of notation, we also indicate by $\nu$ the outer unit normal to the boundary of the annular domain $\partial \Omega(t)$. Then, the variational characterization assures that $u^t\in  H^{1}_{\partial B_{r}(t)}$ is a positive function such that
\begin{equation*}
\sigma(t)=\displaystyle \int_{\Omega_t}|\nabla u^t|^2 \, dx,
\end{equation*}
and that verifies the following equality
\begin{equation}
\label{cn}
\displaystyle\int_{\partial \Omega_0} (u^t)^2 \, d\mathcal H^{n-1}=1.
\end{equation}


\subsection{Differentiability of $\sigma(t)$}

Standard arguments on shape derivatives assure that $u^t$ and $\sigma(t)$ are differentiable with respect to $t$.  For the sake of completeness, we sketchily give the proof, that follows exactly the same arguments contained in \cite{HP}. The main tool is a general version of the implicit function theorem (see \cite{HP,Sc} and also \cite[Lem.2.1]{Bo}) applied to the equation transferred onto the fixed domain $\Omega_0$  (for the details, we refer to \cite{HP} and also to \cite[Lem. 2.7]{Bo} and \cite[Th.1]{HLS}).
\begin{prop} 
Let $\Omega_0$ be as in Definition \ref{set}, $w \in \R^n$ be a unit vector, $B_r(t)$, $\rho_w$, $\Omega(t)$ and $\sigma(t)$ be defined as in \eqref{ph}, \eqref{row}, \eqref{sett} and \eqref{eigt}, respectively. 
Let $u^t$ be the first normalized eigenfunction of $\sigma(t)$, then the functions
\[
t\in[0, \rho_w-r)\to \sigma(t),\qquad t\in[0, \rho_w-r)\to u^t 
\]
are differentiable for any direction $w\in \mathbb S^{n-1}$ and for any $t\in [0, \rho_w-r)$.
\end{prop}
\begin{proof}
Let us fix $t\in [0, \rho_w-r)$ and let $s>0$ be such that $t+s<\rho_w-r$. Therefore, we are able to consider $\sigma(t+s)$ and $u^{t+s}$ and the following weak formulation for problem \eqref{eigt}
holds:
\begin{equation}
\label{weak_f_proof}
\int_{\Omega_0}\nabla u^{t+s}\,\nabla\varphi \,dx=\sigma(t+s)\int_{\partial\Omega_0} u^{t+s}\varphi\, d\mathcal H^
{n-1}\quad \forall\varphi\in H^1_{\partial B_r(t+s)}(\Omega).
\end{equation}
Let $V$ be as in \eqref{campo} and let us define the following map
\[
\Phi:(s,x)\in(-\rho_w+r-t,\rho_w-r-t)\times\Omega_0\to x+sV(x)\in \R^n.
\]
It is easily seen that,
for any $s$, $\Phi(s,\Omega_0)=\Omega_0$, $u^{t+s}(\Phi(s,\cdot)) \in H^1_{\partial B_r(t)}(\Omega)$ and $D\Phi(0,\cdot)=I$, where $D\Phi$ denotes the Jacobian matrix of $\Phi$ and $I$ the identity matrix of order $n$. Therefore, there exists a neighborhood $U$ of $0$ such that $\Phi (s,\cdot)$ is a diffeomorphism of $\Omega_0$ and hence, by making a change of variables, \eqref{weak_f_proof} becomes
\[
\begin{split}
\int_{\Omega_0}\nabla (u^{t+s}\left( \Phi(s,\cdot))\right)(D\Phi(s,\cdot))^{-1})(\nabla(\varphi \left( \Phi(s,\cdot))\right)(D\Phi(s,\cdot))^{-1})|D\Phi(s,\cdot)| \,dx\\
=\sigma(t+s)\int_{\partial\Omega_0}( u^{t+s}\left(\Phi(s,\cdot))\right)(\varphi\left( \Phi(s,\cdot))\right)\, d\mathcal H^{n-1}\quad \forall\varphi\in H^1_{\partial B_r(t+s)}(\Omega),
\end{split}
\]
and the normalization becomes 
\[
\int_{\partial\Omega_0}(u^{t+s}\left( \Phi(s,\cdot))\right)^2 d\mathcal{H}^{n-1}=1.
\]
Now, let us denote by $(H^1_{\partial B_r(t)}(\Omega))'$ the dual space of $H^1_{\partial B_r(t)}(\Omega)$ and let us  define
\[
f:(s,v, \sigma)\in \R\times H^1_{\partial B_r(t)}(\Omega)\times\R\to (f_1,f_2)\in(H^1_{\partial B_r(t)}(\Omega))'\times\R,
\]
where 
\[
\begin{cases}
\langle f_1(s,v,\sigma),\Psi\rangle=\int_{\Omega_0}((\nabla v)(D\Phi(s,\cdot))^{-1})((\nabla \Psi)(D\Phi(s,\cdot))^{-1})|D\Psi (s,\cdot)|dx-\sigma\int_{\Omega_0}v\Psi d\mathcal H^{n-1}\\
\langle f_2(s,v,\sigma),\Psi\rangle=\int_{\partial\Omega_0}v^2d\mathcal H^{n-1},
\end{cases}
\]
for any $\Psi\in H^1_{\partial B_r(t)}(\Omega_0)$.  If we consider the function 
\begin{equation*}
g\colon s \in U \to \left(u^{t+s}(\Phi(s,\cdot),\sigma(t+s)\right) \in H^1_{\partial B_r(t)}(\Omega) \times \mathbb R,
\end{equation*}
then $f(s,g(s))=0$ for any $s\in U$ and  $g(0)=\left(u^t,\sigma(t)\right)$. In order to obtain the claim, we have to prove that $g$ is differentiable in $s=0$, that follows by applying the implicit  function Theorem. To do this, we have to prove that 
\[
\frac{\partial f}{\partial (v,\sigma)}\Big|_{0,u^{t},\sigma(t)}: H^1_{\partial B_r(t)}(\Omega)\times \R\to (H^1_{\partial B_r(t)}(\Omega))'\times\R
\]
is an isomorphism. This can be proved following line by line the same arguments of \cite[Lem 2.7]{Bo} and \cite[Th. 1]{HLS} (see also chapter 5 in \cite{HP}).

\end{proof}
\begin{rem}
We observe that being $u^t$ harmonic for any $t\in[0,\rho_w-r)$, then $u^t \in C^{\infty}\left(\overline{ \Omega(t)}\right)$. Then by the general theory of the shape derivatives (see Chapter 5 of \cite{HP}), follows that $u^t$ is $ C^{\infty}$ in a neighborhood of $t$. 
\end{rem}

Since, to reach our aims, we need to consider the total and partial derivative of $u^t$ with respect to the parameter $t$, we observe that, by \eqref{ph}, we have
\begin{equation}
\label{vel1}
\frac{d}{dt}[u^t(t,x(t))]= (u^t)'+\langle \nabla u^t, w\rangle,
\end{equation}
 where  $(u^t)'=\displaystyle\frac{\partial u^t}{\partial t}$ and $\langle\cdot,\cdot\rangle$ denotes the usual scalar product in $\R^n$. Recalling that the perturbed hole $B_r(t)$ is the zero-level set of the function $u^t$, then \eqref{vel1} implies that
 \begin{equation}
 \label{vel2}
 (u^t)'= -\frac{\partial u^t}{\partial \nu}\, \langle \nu, w\rangle, \quad \text{ on } \partial B_r(t),
 \end{equation} for any $t\in[0,\rho_w-r)$. Moreover, we observe that it holds
\begin{equation}
\label{normal=nabla}
\nu=-\frac{\nabla u^t}{|\nabla u^t|} \quad \text{ on } \partial B_r(t)
\end{equation}
and
\begin{equation*}
(n-1)H=\text{div}\left(\frac{\nabla u^t}{|\nabla u^t|} \right)\quad \text{ on } \partial B_r(t).  
\end{equation*}
Therefore, we have
\begin{equation}
\label{h}
\begin{split}
\frac{(n-1)}{r}=(n-1)H&=-\frac{1}{|\nabla u^t|^2} \langle\nabla\left(|\nabla u^t|\right),\nabla u^t\rangle\\
&=-\frac{1}{|\nabla u^t|^3} \langle\nabla u^t \cdot D^2(u^t),\nabla u^t\rangle\quad \text{ on } \partial B_r(t),
\end{split}
\end{equation}
where $D^2(u^t)$ denotes the Hessian matrix of $u^t$.
Furthermore, by using \eqref{vel2} and the fact that $u^t$ is a solution to \eqref{pt}, we get that $(u^t)'$ is a weak solution to the following problem
 \begin{equation}
 \label{pp}
\begin{cases}
\Delta (u^t)'=0 & \mbox{in}\ \Omega(t)\vspace{0.2cm}\\
\dfrac{\de (u^t)'}{\de \nu}=\sigma' (t)u^t+\sigma(t)(u^t)'&\mbox{on}\ \partial\Omega_0\vspace{0.2cm}\\ 
(u^t)'= -\dfrac{\partial u^t}{\partial \nu}\, \langle \nu, w\rangle&\mbox{on}\ \partial B_{r}(t). 
\end{cases}
\end{equation}
We observe that, if we derive the normalized condition \eqref{cn}, we get
\begin{equation}
\label{cost}
\int_{\partial \Omega_0}(u^t)' \, u^t \, \, d \mathcal H^{n-1}=0.
\end{equation}
Finally, for every $t\in[0,\rho_w-r)$, it will be useful for the sequel to consider the harmonic extension $H(|\nabla u^t| \langle \nu,w\rangle)$ of the function with the same boundary value of \eqref{pp} on $\partial B_r(t)$; that is
\begin{equation}
\label{extension}
\Delta H(|\nabla u^t| \langle \nu,w\rangle)=0\quad \text{in }B_r(t),\qquad H(|\nabla u^t| \langle \nu,w\rangle)=|\nabla u^t| \langle \nu,w\rangle\quad\text{on }\partial B_r(t).
\end{equation}

\section{The first order Shape Derivative of $\sigma(t)$}\label{first}
In this Section, we compute the first oder derivative of the eigenvalue $\sigma(t)$ on $\Omega(t)$.
Before doing this, we recall an Hadamard's formula in the framework of the domain derivative (see for instance \cite{SZ,B,HP}).

Let $E\subset\R^n$ be an open bounded set with Lipschitz boundary and let $V(x)$ a vector field such that $V\in W^{1,\infty}(\R^n;\R^n)$. For any $t>0$, let $E(t)=\{x(t)=x+tV(x), \,x\in E\}$, and $f(t,x(t))$ be such that $f(t,\cdot)\in W^{1,1}(\R^n)$ and differentiable at $t$.
Then it 
  holds:
\begin{gather}\label{hf1}
\begin{split}
\frac{d}{dt} \int_{E(t)}f(t,x)\, d \mathcal H^{n-1}=&
 \int_{E(t)}\frac{\partial}{\partial t}f(t,x) \, d \mathcal H^{n-1} +\\
 & +\int_{+\partial E (t)} f(t,x) \langle \nu,V(x)\rangle\,d \mathcal H^{n-1},
\end{split} 
\end{gather} 
where $\nu$ is the outer unit normal to the boundary of $E(t)$ and $\mathcal H^{n-1}$ denotes the Hausdorff measure (see for instance the Chapter 5 of \cite{HP}). We use formula \eqref{hf1} to prove the following.
\begin{thm} 
\label{der1}
Let $\Omega_0$ be as in Definition \ref{set}, $w \in \R^n$ be a unit vector, $B_r(t)$, $\rho_w$, $\Omega(t)$ and $\sigma(t)$ be defined as in \eqref{ph}, \eqref{row}, \eqref{sett} and \eqref{eigt}, respectively. Then, for any $t\in [0, \rho_w-r)$, it holds
\begin{equation*}
\ds\frac{d}{dt}\sigma(t)=-\displaystyle \int_{\partial B_r(t)}|\nabla u^t|^2 \langle w,\nu \rangle \, d \mathcal H^{n-1},
\end{equation*}
where $u^t$ the  is the normalized,  positive eigenfunction corresponding to $\sigma(t)$.
\end{thm}
\begin{proof}
By using the Hadamard's formula \eqref{hf1}, by observing that the unit outer normals $\nu(x)$ to $\partial B_r$ and $\nu(x(t))$ to $\partial B_r(t)$ coincide and by the fact that $u^t$ is a solution of \eqref{pt}, we get
\begin{gather*}
\begin{split}
\ds\frac{d}{dt}\sigma(t)=& \frac{d}{dt} \displaystyle \left(\int_{\Omega(t)} |\nabla u^t|^2 \, dx \right)\\
=&2\displaystyle \int_{\Omega(t)} \langle\nabla u^t, \nabla (u^t)'\rangle\, dx +\displaystyle \int_{\partial B_r(t)}|\nabla u^t|^2 \langle w,\nu\rangle \, d \mathcal H^{n-1}\\
=& 2  \left(\int_{\partial \Omega_0} (u^t)'\, \frac{\partial u}{\partial \nu} \, d \mathcal H^{n-1} + \int_{\partial B_r(t)} (u^t)'\, \frac{\partial u^t}{\partial \nu} \, d \mathcal H^{n-1} \right)+\\
&+\int_{\partial B_r(t)}|\nabla u^t|^2 \langle w, \nu\rangle \, d \mathcal H^{n-1}\\
=&2  \left(\sigma(t)\int_{\partial \Omega_0} (u^t)'\ \,u \, d \mathcal H^{n-1} - \int_{\partial B_r(t)} \left(\frac{\partial u^t}{\partial \nu}\right)^2\,\langle w,\nu\rangle  d \mathcal H^{n-1} \right) +\\
&+\int_{\partial B_r(t)}|\nabla u^t|^2 \langle w, \nu\rangle \, d \mathcal H^{n-1} .
\end{split}
\end{gather*}
By taking into account the relation \eqref{cost}, we get the conclusion.
\end{proof}
In order to obtain our main result, we study the behavior of the first order derivative of $\sigma(t)$ for $t=0$. We stress that the symmetry of $\Omega_0$ has a key role in the proofs of these results.
\begin{cor}
\label{sim}
Let $\Omega_0$ be as in Definition \ref{set}, $w \in \R^n$ be a unit vector, $B_r(t)$, $\rho_w$, $\Omega(t)$ and $\sigma(t)$ be defined as in \eqref{ph}, \eqref{row}, \eqref{sett} and \eqref{eigt}, respectively. Then it holds
\[
\ds\frac{d}{dt}[\sigma(t)]_{t=0}=0.
\]
\end{cor}
\begin{proof}
Let $u^t$ be  the normalized, positive eigenfunctions corresponding to $\sigma(0)=\sigma(\Omega)$. We observe that, since $\Omega_0$ is centrally symmetric with respect the origin (see Definition \ref{set}), then
\begin{equation}\label{uuprimo}
u^t(x)=u^t(-x)\quad x\in\partial B_r.
\end{equation}
If we denote by $x'$ the symmetric point of $x$, then \eqref{uuprimo} immediately follows by taking $v(x)=u^t(x')$ as test function in \eqref{eig}.


Finally, by Theorem \ref{der1} and  being $\nu= -\displaystyle\frac{x}{r}$, we have
\[
\ds\frac{d}{dt}[\sigma(t)]_{t=0}=-\displaystyle \int_{\partial B_r}|\nabla u^t|^2 \langle \nu,w\rangle \, d \mathcal H^{n-1}=0.
\] 
\end{proof}


\section{The second order Shape Derivative  of $\sigma(t)$}\label{second}
To prove the main result (Theorem \ref{mono}), we need a stationary property of the first order derivative (analyzed in the previous Section) and a sign of the second order derivative. We compute the  second order domain derivative of $\sigma(t)$, by using the same notations of the previous Section and by recalling some useful definitions from \cite{BW}.

For any bounded domain $E\in C^{2,\alpha}$ and for any $f\in C^1(\partial E)$, the tangential derivative of $v$ is given by
\begin{equation}
\label{tangential_grad}
\nabla^\tau f=\nabla f-\langle \nabla f,\nu\rangle\nu\quad\text{on }\partial E,
\end{equation}
where $\nu$ is the unit outer normal. Furthermore, for any smooth vector field $\Psi:\overline E\mapsto\R$, the tangential divergence is defined as
\begin{equation}
\label{tang_div}
\dive_{\partial E}(\Psi)=\dive (\Psi)-\langle\nu,D(\Psi)\nu\rangle \quad\text{on }\partial E,
\end{equation}
where $D(\Psi)$ is the Jacobian matrix of $\Psi$. Then, the Gauss theorem on surfaces holds:
\begin{equation}\label{Gauss_thm}
\int_{\partial E} f\dive_{\partial E} \Psi d\mathcal H^{n-1}=-\int_{\partial E} \langle \Psi, \nabla^\tau f\rangle\  d\mathcal H^{n-1}+(n-1)\int_{\partial E} f\, H\, \langle \Psi,\nu\rangle\ d\mathcal H^{n-1}
\end{equation}
The following Theorem gives the expression of the second order domain derivative of $\sigma(t)$.
\begin{thm} 
\label{der2}
Let $\Omega_0$ be as in Definition \ref{set}, $w \in \R^n$ be a unit vector, $B_r(t)$, $\rho_w$, $\Omega(t)$ and $\sigma(t)$ be defined as in \eqref{ph}, \eqref{row}, \eqref{sett} and \eqref{eigt}, respectively.
Then, for any $t\in [0, \rho_w-r)$, it holds
\begin{equation*}
\begin{split}
\ds\frac{d^2}{dt^2}\sigma(t)&-\frac{2(n-1)}{r}\ds\frac{d}{dt}\sigma(t)=-2\int_{B_r(t)} |\nabla H\left( |\nabla u^t| \langle w,\nu\rangle\right)|^2dx\\
&-\dfrac 1r \int_{\partial B_r(t)}|\nabla u^t|^2d \mathcal H^{n-1}-\frac{n-2}{r}\int_{\partial B_r(t)}|\nabla u^t|^2 \left(\langle w, \nu\rangle \right)^2 \, d \mathcal H^{n-1},
\end{split}
\end{equation*}
where $u^t$ the is the normalized positive eigenfunction corresponding to $\sigma(t)$ and $H(|\nabla u^t| \langle w,\nu\rangle)$ is the harmonic extension of $|\nabla u^t| \langle w,\nu\rangle$ in $\overline{B_r(t)}$, defined in \eqref{extension}.
\end{thm}
\begin{proof}
In order to compute the second order derivative,
we first observe that
\[
\displaystyle \int_{\partial B_r(t)}|\nabla u^t|^2 \langle w,\nu\rangle \, d \mathcal H^{n-1}=\int_{\partial\Omega(t)}|\nabla u^t|^2 \langle V,\nu\rangle \, d \mathcal H^{n-1},
\]
where $V$ is defined in \eqref{campo}. Therefore, by using the divergence Theorem, we have
\[
\sigma'(t)=-\int_{\Omega(t)} \dive ( |\nabla u^t|^2 V)\, d x.
\]

By using the Hadamard's formula \eqref{hf1} in the right-hand side, we compute the second order derivative of $\sigma(t)$:
\begin{gather}
\label{sec2}
\begin{split} 
\sigma''(t)=-\frac{d}{dt} \int_{\Omega(t)} \dive ( |\nabla u^t|^2 V)\, d x=& -\int_{\Omega(t)} \dive \left(\frac{\partial}{\partial t} |\nabla u^t|^2 V\right)\, d x\\
&-\int_{\partial\Omega(t)} \dive ( |\nabla u^t|^2 V)\ \langle V,\nu\rangle\, d \mathcal H^{n-1}.
\end{split}
\end{gather} 
Let us consider the second term in \eqref{sec2}, by using the tangential divergence \eqref{tang_div}, we have
\begin{equation}
\label{second_term}
\begin{split}
&-\int_{\partial\Omega(t)} \dive ( |\nabla u^t|^2 V)\langle V,\nu\rangle\, d \mathcal H^{n-1}=-\int_{\partial B_r(t)} \dive ( |\nabla u^t|^2 w)\langle w,\nu\rangle\, d \mathcal H^{n-1}\\
&\quad=-\int_{\partial B_r(t)} \dive_{\partial B_r(t)} ( |\nabla u^t|^2 w)\langle w,\nu\rangle\, d \mathcal H^{n-1}-\int_{\partial B_r(t)}  \langle \nu,D(|\nabla u^t|^2 w)\nu\rangle \, d \mathcal H^{n-1}.
\end{split}
\end{equation}
Hence, by inserting \eqref{second_term} in \eqref{sec2}, we have 
\begin{equation}
\label{sigma_secondo}
\begin{split}
\sigma''(t)=& -\int_{\Omega(t)} \dive \left(\frac{\partial}{\partial t} |\nabla u^t|^2 V\right)\, d x-\int_{\partial B_r(t)} \dive_{\partial B_r(t)} ( |\nabla u^t|^2 w)\langle w,\nu\rangle\, d \mathcal H^{n-1}\\
&-\int_{\partial B_r(t)}  \langle \nu,D(|\nabla u^t|^2 w)\nu\rangle \, d \mathcal H^{n-1}:=I+II+III.
\end{split}
\end{equation}
For reader's convenience, we separately study the terms $I$, $II$ and $III$ in \eqref{sigma_secondo}. Let us focus on $I$; by using the divergence Theorem, we have
\begin{equation}
\label{I_descr}
\begin{split}
I=-\int_{\Omega(t)} \dive \left(\frac{\partial}{\partial t} |\nabla u^t|^2 V\right)\, d x&=-2\int_{\Omega(t)} \dive \left(\langle \nabla u^t,\nabla (u^t)'\rangle V\right)\, d x\\
&=-2\int_{\partial\Omega(t)}\langle \nabla u^t,\nabla (u^t)'\rangle \ \langle V,\nu\rangle\, d \mathcal H^{n-1}\\
&=-2\int_{\partial B_r(t)}\langle \nabla u^t,\nabla (u^t)'\rangle\ \langle w,\nu\rangle\, d \mathcal H^{n-1}\\
&=-2\int_{\partial B_r(t)} \displaystyle \frac{\partial (u^t)'}{\partial \nu}(u^t)'d \mathcal H^{n-1}\\
&= -2\int_{B_r(t)} |\nabla H\left( |\nabla u^t| \langle w,\nu\rangle\right)|^2dx.
\end{split}
\end{equation}
Let us remark that $\nu$ is also the inner unit normal to $\partial B_r(t)$. 

Now let us consider $II$; by using the Gauss Theorem \eqref{Gauss_thm}, we have
\begin{equation}
\label{II_descr}
\begin{split}
II&=-\int_{\partial B_r(t)} \dive_{\partial B_r(t)} ( |\nabla u^t|^2 w)\langle w,\nu\rangle\, d \mathcal H^{n-1}\\
&=\int_{\partial B_r(t)}\langle |\nabla u^t|^2 w,\nabla^\tau(\langle w,\nu\rangle)\rangle\, d \mathcal H^{n-1}-\frac{n-1}r \int_{\partial B_r(t)}  |\nabla u^t|^2 (\langle w,\nu\rangle)^2\, d \mathcal H^{n-1}\\
&=\int_{\partial B_r(t)}\langle |\nabla u^t|^2 w,\nabla(\langle w,\nu\rangle)-\langle\nabla(\langle w,\nu\rangle),\nu\rangle\nu\rangle\, d \mathcal H^{n-1}\\
&\qquad\qquad\qquad\qquad\qquad\qquad\qquad\qquad\ -\frac{n-1}r \int_{\partial B_r(t)}  |\nabla u^t|^2 (\langle w,\nu\rangle)^2\, d \mathcal H^{n-1}\\
&=\int_{\partial B_r(t)}\left\langle |\nabla u^t|^2 w,-\frac  wr+\left\langle \frac wr,\nu\right\rangle\nu\right\rangle\, d \mathcal H^{n-1}-\frac{n-1}r \int_{\partial B_r(t)}  |\nabla u^t|^2 (\langle w,\nu\rangle)^2\, d \mathcal H^{n-1}\\
&=-\frac 1r \int_{\partial B_r(t)} |\nabla u^t|^2 d \mathcal H^{n-1}+\frac 1r \int_{\partial B_r(t)} |\nabla u^t|^2 (\langle w,\nu\rangle)^2\, d \mathcal H^{n-1}\\
&\qquad\qquad\qquad\qquad\qquad\qquad\qquad\qquad\ -\frac{n-1}r \int_{\partial B_r(t)}  |\nabla u^t|^2 (\langle w,\nu\rangle)^2\, d \mathcal H^{n-1}\\
&=-\frac 1r \int_{\partial B_r(t)} |\nabla u^t|^2 d \mathcal H^{n-1}-\frac{n-2}r \int_{\partial B_r(t)}  |\nabla u^t|^2 (\langle w,\nu\rangle)^2\, d \mathcal H^{n-1},
\end{split}
\end{equation}
where in the third line we have used the tangential gradient \eqref{tangential_grad}.

Let us focus on $III$; we have
\begin{equation}
\label{III_descr}
\begin{split}
III&=-\int_{\partial B_r(t)}  \langle \nu,D(|\nabla u^t|^2 w)\nu\rangle \, d \mathcal H^{n-1}=-2\int_{\partial B_r(t)}  \langle \nu,  D^2(u^t)\nabla u^t \rangle\ \langle w,\nu\rangle \, d \mathcal H^{n-1}\\
&=2\int_{\partial B_r(t)}  \dfrac{1}{|\nabla u^t|^3}\langle \nabla u^t,  D^2(u^t)\nabla u^t \rangle|\nabla u^t|^2 \langle w,\nu\rangle \, d \mathcal H^{n-1}\\
&=-\dfrac{2(n-1)}{r}\int_{\partial B_r(t)} |\nabla u^t|^2 \langle w,\nu\rangle \, d \mathcal H^{n-1},
\end{split}
\end{equation}
where we have used the relations \eqref{normal=nabla} and \eqref{h}.
The conclusion follows by using \eqref{I_descr}, \eqref{II_descr} and \eqref{III_descr} in \eqref{sigma_secondo}.
\end{proof}

Finally, the proof of the main Theorem is a direct consequence of the computed expressions of the first and second order shape derivative of the first Steklov-Dirichlet eigenvalue.

\begin{proof}[Proof of the Theorem \ref{mono}]
Since $n\geq 2$, the claim is a direct consequence of  the Theorems \ref{der1}, Corollary \ref{sim} and Theorem \ref{der2}  
 \end{proof} 
 \begin{rem}
An  immediate consequence of Theorem \ref{mono} is that  $\sigma(t)$ is maximum when the  hole is centered at the  symmetry point of $\Omega_0$, when $t=0$. 
\end{rem}
 \begin{rem}
 We stress that, when $\Omega_0=B_R$, the authors in \cite{HLS} prove the following estimate in two dimensions for $R>r$:
 \begin{equation}
 \label{lb}
 \liminf_{t \to (R-r)^-} \sigma(t) \ge \displaystyle \frac{r}{2R(R-r)},
 \end{equation}
 where $\sigma(t)=\sigma(\Omega(t))$ and $\Omega(t)=B_{R}\setminus B_{r}(t)$.
 By Theorem \ref{mono}, the estimate \eqref{lb} can be written as the following lower bound  
 \begin{equation}
 \label{st}
 \frac{1}{R\log(\frac{R}{r})}=\sigma(B_R\setminus B_r) \ge \sigma(t) \ge  \displaystyle \frac{r}{2R(R-r)}, \quad \forall t \in[0,\rho_w-r).
 \end{equation}
 Finally, we observe that the inequality \eqref{st} gives an upper and lower bound for $\sigma(t)$ in terms of the two radius of the eccentric annulus.
 \end{rem}
\section*{Acknowledgements}
This work has been partially supported by the MiUR-PRIN 2017 grant \lq\lq Qualitative and quantitative aspects of nonlinear PDEs\rq\rq, by GNAMPA of INdAM and by FRA 2020 \lq\lq Optimization problems in Geometric-functional inequalities and nonlinear PDEs\rq\rq(OPtImIzE).

We would like to thank the reviewer for his/her suggestions to improve the paper.

\section*{Compliance with Ethical Standards}
This paper does not disclose of potential conflicts of interest.

\small{

}


\begin{thebibliography}{999}
		
\bibitem[A]{A} M. S. Agranovich. \textit{On a mixed Poincar\'e-Steklov type spectral problem in a Lipschitz domain}. Russ. J. Math. Phys. 13.3 (2006): 239-244.

\bibitem[AA]{AA} T. V. Anoop, K. Ashok Kumar.  \textit{Domain variations of the first eigenvalue via a strict Faber-Krahn type inequality}. ArXiv (2022): 1-21.

\bibitem[AAK]{AAK} T. V. Anoop, K. Ashok Kumar, S. Kesavan. \textit{A shape variation result via the geometry of eigenfunctions}. J. Differential Equations 298 (2021): 430-462.

		
\bibitem[BW]{BW} C. Bandle, A. Wagner. {\it Second domain variation for problems with Robin boundary conditions}. J. Optim. Theory Appl. 167.2  (2015): 430-463.
		
		
	
\bibitem[B]{B} G. Bellettini. \textit{Lecture notes on mean curvature flow: barriers and singular perturbations}. Springer (2014): 327 pp.

\bibitem[Bo]{Bo} J. F. Bonder, P. Groisman, J. D. Rossi. \textit{Optimization of the first Steklov eigenvalue in domains with holes: a shape derivative approach}. Ann. Mat. Pura Appl. (4) 186.2 (2007): 341-358.


 



  
  
  

\bibitem[DP]{DP} F. Della Pietra, G. Piscitelli. \textit{An optimal bound for nonlinear eigenvalues and torsional rigidity on domains with holes}. Milan J. Math. 88.2 (2020): 373-384.

\bibitem[D]{D} B. Dittmar. \textit{Eigenvalue problems and conformal mapping, Handbook of complex analysis: geometric function theory. Vol. 2}, Elsevier Sci. B. V., Amsterdam (2005): 669-686.
		

		
\bibitem[ET]{ET} M. Egert, P. Tolksdorf. \textit{Characterizations of Sobolev functions that vanish on a part of the boundary}. Discrete Contin. Dyn. Syst. Ser. S 10.4 (2017): 729-743.
		
		

\bibitem[F]{F} I. Ftouhi. \textit{Where to place a spherical obstacle so as to maximize the first Steklov eigenvalue}. ESAIM: Control Optim. Calc. Var. 28.6 (2022): 1-21.




\bibitem[GPPS]{GPPS} N. Gavitone, G. Paoli, G. Piscitelli, R. Sannipoli. \textit{An Isoperimetric inequality for the first  Steklov-Dirichlet Laplacian eigenvalue of convex sets with a spherical hole}. Pacific J. Math. 320.2 (2022): 241-259.

\bibitem[GS]{GS} N. Gavitone, R. Sannipoli. \textit{On a Steklov-Robin eigenvalue problem}.  J. Math. Anal. Appl. (in press).


\bibitem[HP]{HP} A. Henrot, M. Pierre. \textit{Shape variation and optimization}. EMS Tracts in Mathematics 28 (2018): 401 pp.


\bibitem[HLS]{HLS} J. Hong, M. Lim, D.H. Seo.  \textit{On the first Steklov?Dirichlet eigenvalue for eccentric annuli}. Ann. Mat. Pur. Appl. (1923-) 201.2 (2022): 769-799.





\bibitem[L]{L} L. Li. \textit{On the placement of an obstacle so as to optimize the Dirichlet heat content}. SIAM J. Math. Anal. 54.3 (2022): 3275-3291.


\bibitem[P]{P} B. V. E. Pal'tsev. \textit{Mixed problems with non-homogeneous boundary conditions in Lipschitz domains for second-order elliptic equations with a parameter}. Sb. Math. 187.4 (1996): 525.

\bibitem[PPS]{PPS} G. Paoli, G. Piscitelli, R. Sannipoli. \textit{A stability result for the Steklov Laplacian Eigenvalue Problem with a spherical obstacle}. Comm. Pure Appl. Anal. 20.1 (2021): 145-158.

\bibitem[PPT]{PPT} G. Paoli, G. Piscitelli, L. Trani, \textit{Sharp estimates for the first $p$-laplacian eigenvalue and for the $p$-torsional rigidity on convex sets with holes}. ESAIM Control Calc. Var. 26.111 (2020): 1-15.

\bibitem[S]{S}D.H. Seo. \emph{A shape optimization problem for the first mixed Steklov-Dirichlet eigenvalue.} Ann. Glob. Anal. Geom. 59.3: 345-365.
		
\bibitem[Sc]{Sc} L. Schwartz, \emph{Cours d'analyse. 1 et 2}. Second edition. Hermann, Paris (1981).

				
			
\bibitem[SZ]{SZ} J. Sokolowski, P. Zol\'esio. \textit{Introduction to shape optimization}. In: Introduction to shape optimization. Springer, Berlin, Heidelberg (1992): 5-12.
			
			
\bibitem[VS]{VS} S. Verma, G. Santhanam. \textit{On eigenvalue problems related to the laplacian in a class of doubly connected domains}. Monatsh. Math. 193 (2020): 879-899.
	

	
\end{thebibliography}
\end{document}